\documentclass[12pt]{amsart}
\usepackage{amssymb,amsmath}
\usepackage{geometry}    
\usepackage{mathrsfs}

\usepackage{vmargin}
\setmargrb{1in}{1in}{1in}{1in}

\newtheorem{theorem}{Theorem}[section]
\newtheorem{lemma}[theorem]{Lemma}
\newtheorem{proposition}[theorem]{Proposition}
\newtheorem{corollary}[theorem]{Corollary}
\theoremstyle{definition}
\newtheorem{definition}[theorem]{Definition}
\newtheorem{remark}[theorem]{Remark}

\newtheorem{conjecture}[theorem]{Conjecture}
\newtheorem{problem}[theorem]{Problem}

\begin{document}

\title{On the spectrum of Diophantine approximation constants}

\author{Johannes Schleischitz} 

\address{Institute of Mathematics, Univ. Nat. Res. Life Sci. Vienna, Austria}

\begin{abstract}
The approximation constant $\lambda_{k}(\zeta)$ is defined as the supremum of 
$\eta\in{\mathbb{R}}$ such that the estimate $\max_{1\leq j\leq k}\Vert \zeta^{j}x\Vert\leq x^{-\eta}$ 
has infinitely many integer solutions $x$. Here $\Vert.\Vert$ denotes the distance
to the closest integer.
We establish a connection on the joint spectrum $(\lambda_{1}(\zeta),\lambda_{2}(\zeta),\ldots)$,
which will lead to various
improvements of known results on the individual spectrum of the approximation constants $\lambda_{k}(\zeta)$
as well. In particular, for given $k\geq 1$ and $\lambda\geq 1$, we
construct $\zeta$ in the Cantor set with $\lambda_{k}(\zeta)=\lambda$. 
Moreover, we establish an estimate for the uniform approximation constants 
$\widehat{\lambda}_{k}(\zeta)$, which enables us to determine 
classical approximation constants for Liouville numbers.
\end{abstract}

\maketitle

{\footnotesize{Supported by the Austrian Science Fund FWF grant P24828.} \\

{\em Keywords}: Diophantine approximation, approximation constants, Hausdorff dimension, continued fractions \\
Math Subject Classification 2010: 11H06, 11J13, 11J25, 11J82, 11J83}

\vspace{4mm}

\section{Introduction and main results}  

\subsection{Definition of the constants}  \label{intro}
We begin with the definition of the quantities $\lambda_{k}(\zeta)$
that we will predominately consider and their uniform versions $\widehat{\lambda}_{k}(\zeta)$.
For a vector $\underline{\zeta}=(\zeta_{1},\ldots,\zeta_{k})\in{\mathbb{R}^{k}}$ 
we denote by $\lambda_{k}(\underline{\zeta})$ and $\widehat{\lambda}_{k}(\underline{\zeta})$, respectively, 
the supremum of $\eta\in{\mathbb{R}}$ such that the system 
\begin{equation}  \label{eq:lambda}
\vert x\vert \leq X, \qquad 0<\max_{1\leq j\leq k} \vert \zeta_{j}x-y_{j}\vert \leq X^{-\eta},  
\end{equation}
has a solution $(x,y_{1},y_{2},\ldots, y_{k})\in{\mathbb{Z}^{k+1}}$ 
for arbitrarily large $X$ and for all $X\geq X_{0}$, respectively. 
Furthermore, let
\[
\lambda_{k}(\zeta):=\lambda_{k}(\zeta,\zeta^{2},\ldots,\zeta^{k}), \qquad
\widehat{\lambda}_{k}(\zeta):=\widehat{\lambda}_{k}(\zeta,\zeta^{2},\ldots,\zeta^{k}).
\]
In the special case $\zeta\in{\mathbb{Q}}$, it is not difficult to see that due to the 
non-vanishing condition in \eqref{eq:lambda}, we have 
$\lambda_{k}(\zeta)=\widehat{\lambda}_{k}(\zeta)=0$ for all $k\geq 1$.
If $\zeta$ is algebraic of degree $d\geq 2$, then 
$\lambda_{k}(\zeta)=\widehat{\lambda}_{k}(\zeta)=\max\{1/k,1/(d-1)\}$
is a consequence of the Schmidt~Subspace~Theorem~\cite{subspace}, as pointed out in~\cite{bug}.
Assuming otherwise that $\zeta$ is not algebraic of degree $\leq k$,
Dirichlet's box principle implies 
\begin{equation} \label{eq:dirichlet}
\frac{1}{k}\leq \widehat{\lambda}_{k}(\zeta)\leq \lambda_{k}(\zeta)\leq \infty.
\end{equation}
In fact, \eqref{eq:dirichlet} holds for any irrational real number $\zeta$.
The analogue result holds for
$\underline{\zeta}\in{\mathbb{R}^{k}}$ for which $\{1,\zeta_{1},\zeta_{2},\ldots,\zeta_{k}\}$
is $\mathbb{Q}$-linearly independent as well. However, 
the remaining results of Section~\ref{intro} require the assumption 
$\underline{\zeta}=(\zeta,\zeta^{2},\ldots,\zeta^{k})$, and some
cannot be defined in the general case anyway. 

It follows from the definitions that for any $\zeta\in{\mathbb{R}}$ we have the chains of inequalities
\begin{eqnarray} 
&\cdots&\leq \lambda_{3}(\zeta)\leq \lambda_{2}(\zeta)\leq \lambda_{1}(\zeta), \label{eq:monoton}      \\
&\cdots&\leq \widehat{\lambda}_{3}(\zeta)\leq \widehat{\lambda}_{2}(\zeta)\leq \widehat{\lambda}_{1}(\zeta). 
\label{eq:monodon}
\end{eqnarray}
It is not hard to construct $\zeta$ such that 
the asymptotic constants $\lambda_{k}(\zeta)$ take the value $\infty$
(which for a given $\zeta$ holds for all $k$ simultaneously or for none by Corollary~2 in~\cite{bug}),
which one may deduce from Theorem~\ref{meinsatz} in Section~\ref{deandre}. 
The uniform constants, on the other hand, can be effectively bounded.
An elementary result by Khintchine~\cite{khintchine} implies that for any irrational $\zeta$,
the one-dimensional uniform approximation constant is given as
\begin{equation}  \label{eq:uniformity}
\widehat{\lambda}_{1}(\zeta)=1.
\end{equation}
For $k=2$ and $\zeta$ neither rational nor quadratic irrational, we have
\[
\widehat{\lambda}_{2}(\zeta)\leq  \frac{\sqrt{5}-1}{2},
\]
and this constant is optimal. Equality holds for so-called extremal numbers
that can be explicitly constructed. See~\cite{roy},~\cite{royy},~\cite{royyy}.
As usual $\lceil \alpha\rceil$ denotes
the smallest integer greater or equal than $\alpha\in{\mathbb{R}}$.
For any $k\geq 2$ and $\zeta$ not algebraic of degree $\leq \lceil k/2\rceil$, the upper bound
\begin{equation}  \label{eq:beides}
\widehat{\lambda}_{k}(\zeta)\leq \frac{1}{\left\lceil \frac{k}{2}\right\rceil}
\end{equation}
is known, which is a slight refinement due to Laurent~\cite{laurent} of a slightly weaker result
established in~\cite{schmidt1969}. A refinement has been made for $k=3$ by Roy~\cite{roy2}. 
The bound in \eqref{eq:beides} is not considered to be optimal for any $k\geq 2$. 

\subsection{Problems from \cite{bug} and partial results} \label{partres}
We now introduce the problems stated as Problem~1, Problem~2, and Problem~3 
in~\cite{bug}, upon which we will focus.

\begin{problem}  \label{problem1}
Let $k$ be a positive integer. Is the spectrum of the function $\lambda_{k}$ 
equal to $[1/k,\infty]$?
\end{problem}

\begin{problem}  \label{problem2}
Let $k$ be a positive integer and $\lambda\geq 1/k$. Determine the Hausdorff dimensions
\begin{equation} \label{eq:abc}
\dim(\{\zeta\in{\mathbb{R}}: \lambda_{k}(\zeta)=\lambda\}), 
\quad \dim(\{\zeta\in{\mathbb{R}}: \lambda_{k}(\zeta)\geq \lambda\}).
\end{equation}
\end{problem}

Concerning the third problem one needs to know that 
\begin{equation} \label{eq:holds}
\lambda_{kn}(\zeta)\geq \frac{\lambda_{n}(\zeta)-k+1}{k}, \qquad n\geq 1, k\geq 1,
\end{equation}
which was proved by~Bugeaud~\cite[Lemma1]{bug}. 

\begin{problem} \label{problem3}
Let $(\lambda_{k})_{k\geq 1}$ be a sequence of non-increasing real numbers satisfying
\[
\lambda_{k}\geq \frac{1}{k}, \qquad \quad k\geq 1,
\]
and 
\[
\lambda_{kn}\geq \frac{\lambda_{n}-k+1}{k}, \qquad n\geq 1, k\geq 1.
\]
Does there exist a real number $\zeta$ with $\lambda_{k}(\zeta)=\lambda_{k}$ for all $k\geq 1$?
\end{problem}

Next we recall some notes and known partial results on the problems that can 
be found in~\cite{bug} as well. We remark the fact that in these results
both sides in \eqref{eq:abc} will 
coincide, which is not explicitly mentioned in~\cite{bug}. 
This can be inferred from the original results
with a standard argument on dimension functions.   

A positive answer to Problem~\ref{problem1} has been established only for $k\in{\{1,2\}}$. 
In this case it is a consequence of the subsequent results we quote.
Until now, for general $k$, it has been shown in~\cite[Theorem~2]{bug}
that the spectrum contains the interval $[1,\infty]$. Nevertheless, we will not improve Problem~\ref{problem1}.
However, our methods will allow for improving~\cite[Theorem~BL]{bug} 
on the spectrum of $\lambda_{k}(\zeta)$
among $\zeta$ restricted to Cantor's middle third set, see Section~\ref{cantorset}.

We turn to the metrical results. For $k=1$, the following result 
concerning Problem~\ref{problem2} is due to Jarn\'ik~\cite{jar}.

\begin{theorem}[Jarn\'ik]  \label{vjarnik}
Let $k=1$ and $\lambda\geq 1$ be a real number. Then both sides 
in \eqref{eq:abc} equal $2/(1+\lambda)$.
\end{theorem}

For $k\geq 2$ the following is known~\cite{bu}.

\begin{theorem}[Budarina, Dickinson, Levesley]  \label{verbessern}
Let $k\geq 2$ be an integer and $\lambda\geq k-1$. Then both sides 
in \eqref{eq:abc} equal $2/(k(1+\lambda))$.
\end{theorem}

Problem~\ref{problem2} for the sets in \eqref{eq:abc} was solved for $k=2$ as well by
Beresnevich, Dickinson, Vaughan and Velani~\cite{bere},~\cite{vel}. Indeed,
in the case $1/2\leq \lambda< 1$ not contained in Theorem~\ref{verbessern},
they prove that
\begin{equation} \label{eq:bb}
\dim(\{\zeta\in{\mathbb{R}}: \lambda_{2}(\zeta)= \lambda\})= 
\dim(\{\zeta\in{\mathbb{R}}: \lambda_{2}(\zeta)\geq \lambda\})= \frac{2-\lambda}{1+\lambda}.
\end{equation}
It is worth noting that for any integer $k\geq 2$ and $\lambda\geq 1/k$ we have
\begin{equation} \label{eq:metrisch}
\dim(\{\zeta\in{\mathbb{R}}: \lambda_{k}(\zeta)\geq \lambda\})\geq \frac{2}{k(1+\lambda)}.
\end{equation}
For small values of $\lambda$ and $k\geq 3$ a metric result due to Beresnevich 
can be found in~\cite{beres}.
However, \eqref{eq:metrisch} does not even allow for answering Problem~\ref{problem1} for $k\geq 3$.
As a last reference concerning Problem~\ref{problem2} we want to point 
out a consequence of a famous result of Sprind\^zuk~\cite{sprindzuk}, that
the set of $\zeta\in{\mathbb{R}}$ with $\lambda_{k}(\zeta)>1/k$ has Lebesgue measure $0$.

Concerning Problem~\ref{problem3}, for now we solely mention that
it is based on the fact that each of the assumptions is known to hold for 
any real number $\zeta$. We quoted the origin of \eqref{eq:holds}, the other 
restrictions are due to \eqref{eq:dirichlet}, \eqref{eq:monoton}.

\subsection{Outline of new results on the problems}

Our methods yield a converse to \eqref{eq:holds}, provided that the 
constants $\lambda_{k}$ are strictly greater than $1$. This is the 
key for a better understanding of the Problems~\ref{problem2},~\ref{problem3}.
We prove the following.
 
\begin{theorem}  \label{spektrum}
Let $k$ be a positive integer and $\zeta$ be a real number such that $\lambda_{k}(\zeta)>1$.
Then 
\[
\lambda_{k}(\zeta)= \frac{\lambda_{1}(\zeta)-k+1}{k}.
\]
Moreover, for $1\leq j\leq k$ we have $\widehat{\lambda}_{j}(\zeta)=1/j$.
\end{theorem}

\begin{remark} \label{rem1}
Note that the restriction $\lambda_{k}>1$ is necessary. Theorem~4.3 in~\cite{bug} shows
that there exist $\zeta$ such that $\lambda_{k}(\zeta)=1$ for all $k\geq 1$. Theorem~4.4
in~\cite{bug} shows that for $1\leq \lambda \leq 3$, there exist $\zeta$ with
$\lambda_{1}(\zeta)=\lambda$ and $\lambda_{2}(\zeta)=1$. This shows that in general, $\lambda_{1}(\zeta)$
is not determined by $\lambda_{2}(\zeta)$ as in Theorem~\ref{spektrum}. 
\end{remark}

We now present several corollaries to Theorem~\ref{spektrum}. Concerning Problem~\ref{problem2},
we deduce from Theorem~\ref{spektrum} the assertion of Theorem~\ref{verbessern} 
for a larger class of parameters $\lambda$.

\begin{corollary}  \label{ergebnis}
Let $k\geq 2$ be an integer and $\lambda$ be a parameter. For $\lambda>1$, both sides
in \eqref{eq:abc} equal $2/(k(1+\lambda))$. For $\lambda=1$,
both dimensions in \eqref{eq:abc} coincide and are at least $1/k$.
\end{corollary}

\begin{proof}
For $\lambda>1$, Theorem~\ref{spektrum} implies the set identity
\begin{equation} \label{eq:neue}
\{\zeta\in{\mathbb{R}}: \lambda_{k}(\zeta)=\lambda\}=\{\zeta\in{\mathbb{R}}: \lambda_{1}(\zeta)=k\lambda+k-1\}, 
\end{equation}
and similarly for the right hand sides in \eqref{eq:abc}.
In particular the corresponding Hausdorff dimensions coincide. 
The dimension of the right hand side in \eqref{eq:neue} (resp. for $\geq$)
can be determined with Theorem~\ref{vjarnik}, which gives just $2/(k(1+\lambda))$.

For $\lambda=1$, note that in view of \eqref{eq:holds}, the equality $\lambda_{1}(\zeta)=2k-1$ implies
$\lambda_{k}(\zeta)\geq 1$. However, $\lambda_{k}(\zeta)>1$ is impossible due to Theorem~\ref{spektrum}. 
Thus
\[
\{\zeta\in{\mathbb{R}}: \lambda_{k}(\zeta)=1\}\supset\{\zeta\in{\mathbb{R}}: \lambda_{1}(\zeta)=2k-1\},
\]
and the lower bound $1/k$ follows from Theorem~\ref{vjarnik} again. 
It remains to be shown that both sides in \eqref{eq:abc} coincide for $\lambda=1$ as well.
The standard argument on dimension functions mentioned subsequent to Problem~\ref{problem3}
and the results in~\cite{jar} show that actually the $1/k$-dimensional Hausdorff measure of the set
$\{\zeta: \lambda_{1}(\zeta)>2k-1\}$ is $0$. On the other hand 
$\{\zeta: \lambda_{k}(\zeta)>1\}=\{\zeta: \lambda_{1}(\zeta)>2k-1\}$ by virtue of Theorem~\ref{spektrum}.
The claim follows.
\end{proof}

Note that \eqref{eq:bb} suggests that we cannot expect 
Corollary~\ref{ergebnis} to hold for $\lambda<1$. 
We turn towards Problem~\ref{problem3}. First we discuss the case of Liouville numbers,
i.e. real numbers with $\lambda_{1}(\zeta)=\infty$, separately. 
In Section~\ref{intro} it was pointed out that in our setting
$\zeta\in{\mathbb{Q}}$ satisfies $\lambda_{1}(\zeta)=0$,
such that Liouville numbers are irrational. 
They are in fact transcendental by Liouville's Theorem (or Roth's Theorem, see Section~\ref{rothw}). 
Recall from Section~\ref{intro} that for Liouville numbers in fact $\lambda_{k}(\zeta)=\infty$ for all 
$k\geq 1$ by virtue of \eqref{eq:holds}, and on the other hand if $\zeta\in{\mathbb{Q}}$
then $\lambda_{k}(\zeta)=0$ for $k\geq 1$.
Thus we may restrict to the case of irrational $\zeta$ with $\lambda_{1}(\zeta)<\infty$.
The following Corollary~\ref{bestens} to Theorem~\ref{spektrum} shows that in this case all
$\lambda_{.}(\zeta)>1$ are determined by $\lambda_{1}(\zeta)$ as well, and there can only be a 
finite number of them. This readily implies a negative answer to Problem~\ref{problem3}.

\begin{corollary}  \label{bestens}
Let $\zeta$ be an irrational real number with $\lambda_{1}(\zeta)<\infty$
and define $k_{0}:=\lceil (\lambda_{1}(\zeta)+1)/2\rceil$. Then
\begin{eqnarray}
&\lambda_{k}(\zeta)&=\frac{\lambda_{1}(\zeta)-k+1}{k}, \qquad 1\leq k\leq k_{0}-1, \label{eq:systemm1}  \\
\max\left\{\frac{\lambda_{1}(\zeta)-k+1}{k},\frac{1}{k}\right\}\leq& \lambda_{k}(\zeta)&\leq 1, 
\qquad \qquad \qquad \qquad \quad k\geq k_{0}.   \label{eq:systemm2}
\end{eqnarray}
In particular, there are only finitely many indices $k$ with
$\lambda_{k}(\zeta)>1$.
Moreover, the answer to Problem~{\upshape\ref{problem3}} is no.
\end{corollary}

\begin{proof}
Note that for all integers $k<k_{0}$, by construction $(\lambda_{1}(\zeta)-k+1)/k>1$ holds.
By \eqref{eq:holds} for $n=1$, we have $\lambda_{k}(\zeta)>1$ such that
we can apply Theorem~\ref{spektrum} to obtain the reverse inequalities in \eqref{eq:systemm1}
for all $k< k_{0}$. The left inequalities in \eqref{eq:systemm2} are due to 
\eqref{eq:holds} and \eqref{eq:dirichlet}. (Observe that \eqref{eq:dirichlet} holds for
irrational algebraic numbers too, as pointed out in Section~\ref{intro}.) 
The right hand side inequality follows again from the definition of $k_{0}$
and Theorem~\ref{spektrum}.

For the last assertion, put $\lambda_{1}=\lambda_{2}=1.1$ and $\lambda_{k}=1/k$ for $k\geq 3$.
One readily sees that the conditions of Problem~\ref{problem3} are satisfied, but the 
choice contradicts Theorem~\ref{spektrum}.
\end{proof}

In view of Remark~\ref{rem1}, the finiteness assertion is wrong if we relax
the assumption to $\lambda_{k}(\zeta)\geq 1$.
A rearrangement of Corollary~\ref{bestens} allows for determining
$\lambda_{m}(\zeta)$ from $\lambda_{n}(\zeta)$ for arbitrary indices $m,n$,
provided that both values are strictly larger than one. 
In particular, this implies equality in \eqref{eq:holds} for $n>1$ as well. 
This is part of the following corollary.

\begin{corollary} \label{besten}
Let $1\leq n\leq m$ be integers and $\zeta\in{\mathbb{R}}$.
If $\lambda_{n}(\zeta)>1$, then we have the inequality
\[
\lambda_{m}(\zeta)\geq \frac{n\lambda_{n}(\zeta)+n-m}{m}.
\]
If, moreover, $\lambda_{m}(\zeta)>1$, then 
\[
\lambda_{m}(\zeta)=\frac{n\lambda_{n}(\zeta)+n-m}{m}.
\]
In particular, choosing $m=kn$, we obtain equality in {\upshape(\ref{eq:holds})}.
\end{corollary}

\begin{proof}
Assuming $\lambda_{n}(\zeta)>1$, Theorem~\ref{spektrum} implies $(\lambda_{1}(\zeta)-n+1)/n=\lambda_{n}(\zeta)$.
On the other hand, \eqref{eq:holds} yields $\lambda_{m}(\zeta)\geq (\lambda_{1}(\zeta)-m+1)/m$. The assertion 
follows from basic rearrangements.

Assuming $\lambda_{m}(\zeta)>1$,
since $n\leq m$ implies $\lambda_{n}(\zeta)\geq \lambda_{m}(\zeta)>1$, we may apply 
Theorem~\ref{spektrum} to both indices. This yields $(\lambda_{1}(\zeta)-n+1)/n=\lambda_{n}(\zeta)$ 
and $(\lambda_{1}(\zeta)-m+1)/m=\lambda_{m}(\zeta)$. 
It is easy to deduce the identity. 
\end{proof}

The first assertion of Corollary~\ref{besten} motivates the question of whether the assumption $\lambda_{n}(\zeta)>1$
is necessary or not. If not, we would obtain an ultimate generalization of \eqref{eq:holds}.

\begin{problem}
Let $m\geq n\geq 1$ be integers. Does the estimate
\begin{equation} \label{eq:giltdaswohl}
\lambda_{m}(\zeta)\geq \frac{n\lambda_{n}(\zeta)+n-m}{m}
\end{equation}
hold for any $\zeta\in{\mathbb{R}}$?
\end{problem}

Dropping the assumption $m\geq n$, counterexamples are provided by
taking $m=1, n\geq 2$ and any $\zeta$ with $\lambda_{1}(\zeta)<2$,
noting $n\lambda_{n}(\zeta)\geq 1$ by \eqref{eq:dirichlet}. Note 
also that \eqref{eq:holds} is trivial in the case of $\lambda_{n}\leq 1$,
whereas \eqref{eq:giltdaswohl} is generally not.
We check that the smallest pair $(m,n)$ that leads to a non-trivial case of \eqref{eq:giltdaswohl} 
is $m=4, n=3$. We can assume $\lambda_{n}(\zeta)\leq 1$, 
otherwise Corollary~\ref{besten} applies. Now
indeed, in view of $\lambda_{n}(\zeta)\leq 1$, for $m=3, n=2$ 
the estimate \eqref{eq:giltdaswohl} is trivially implied by \eqref{eq:dirichlet}. 
For $m\geq 5, n=2$ it turns out to be trivial anyway, and all other 
pairs $(m,n)$ with either $m\leq 4$ or $n\leq 3$ or both yield a special case of
\eqref{eq:holds}. For $m=4, n=3$, if $\lambda_{3}(\zeta)\in{(2/3,1]}$, then 
\eqref{eq:giltdaswohl} would lead to some bound $\lambda_{4}(\zeta)>\eta>1/4$,
where $1/4$ is the trivial lower bound from \eqref{eq:dirichlet}. 

Similar to the proof of Lemma~\ref{lemma2}, which is the main tool for 
the proof of Theorem~\ref{spektrum}, we will derive the following result in Section~\ref{beweis}.

\begin{theorem} \label{theo}
Let $k$ be a positive integer and $\zeta$ be a real number. Then
\[
\widehat{\lambda}_{k}(\zeta)\leq \max\left\{\frac{1}{k},\frac{1}{\lambda_{1}(\zeta)}\right\}.
\]
\end{theorem}

The assertion of Theorem~\ref{theo}, for $\zeta$ not algebraic of degree $\leq \lceil k/2\rceil$,
is of interest only in case of $\lambda_{1}(\zeta)>\lceil k/2\rceil$ due to \eqref{eq:beides}. 
Also note that Theorem~\ref{theo} improves the assertion on $\widehat{\lambda}_{j}, 1\leq j\leq k$ 
from Theorem~\ref{spektrum}, since $\lambda_{k}(\zeta)>1$ is a stronger assumption 
than $\lambda_{k}(\zeta)\geq k$ by Theorem~\ref{spektrum}.
See also Theorem~\ref{neuko} in Section~\ref{rothw}.

The remainder of the paper is organized as follows. In Section~\ref{beginn} 
we gather preparatory results for the proofs of Theorems~\ref{spektrum} and~\ref{theo},
which will be carried out in Section~\ref{beweis}. In Section~\ref{cantorset},
we study the simultaneous rational approximation properties of numbers in special fractal sets,
such as Cantor's middle third set, as already indicated in Section~\ref{partres}.
Finally, in Section~\ref{vier} we discuss the consequences of our results for the well-studied 
related constants $w_{k}(\zeta),\widehat{w}_{k}(\zeta)$ dealing with polynomial approximation.

\section{Preparatory results}  \label{beginn}

In this section we gather technical lemmas for the proofs of Theorem~\ref{spektrum} and \ref{theo}.
The lemmas will be applied in the proof of Theorem~\ref{meinsatz} in Section~\ref{cantorset} as well.
First we need the rather elementary Lemma~\ref{help} concerning one-dimensional Diophantine approximation.
It can be proved using elementary facts of continued fraction expansions.
In the sequel we will denote by $\Vert \alpha\Vert$ the distance of $\alpha\in{\mathbb{R}}$
to the nearest integer as usual.
  
\begin{lemma} \label{help}
Let $\zeta\in{\mathbb{R}}$. Suppose that for a positive integer $x$ we have the estimate
\begin{equation} \label{eq:eins}
\Vert \zeta x\Vert < \frac{1}{2}x^{-1}.
\end{equation}
Then there exist positive integers $x_{0},y_{0},M_{0}$ such that
$x=M_{0}x_{0}$, $(x_{0},y_{0})=1$ and
\begin{equation} \label{eq:leichter}
\vert \zeta x_{0}-y_{0}\vert=\Vert \zeta x_{0}\Vert= \min_{1\leq v\leq x} \Vert \zeta v\Vert.   
\end{equation}
Moreover, we have the identity
\begin{equation} \label{eq:multiplizieren}
\Vert \zeta x\Vert=M_{0}\Vert \zeta x_{0}\Vert.
\end{equation}
The integers $x_{0},y_{0},M_{0}$ are uniquely determined by the fact that
$y_{0}/x_{0}$ is the convergent {\upshape(}in lowest terms{\upshape)} 
of the continued fraction expansion of $\zeta$
with the largest denominator not exceeding $x$, and $M_{0}=x/x_{0}$.
\end{lemma}

\begin{proof}
Let $y$ be the integer for which $\vert \zeta x-y\vert < (1/2)x^{-1}$
holds. It is a well-known result on continued fractions~\cite[Satz~11]{perron} 
that $y/x$ must be a convergent of $\zeta$.
Let $M_{0}$ be the greatest common divisor of
$x,y$, such that $(x,y)=(M_{0}x_{0},M_{0}y_{0})$ for coprime integers $x_{0},y_{0}$
and $y_{0}/x_{0}$ a convergent of $\zeta$ in lowest terms.
Then \eqref{eq:multiplizieren} follows from
\begin{equation} \label{eq:hilfsug}
\Vert \zeta x\Vert=\vert \zeta x-y\vert = \vert \zeta M_{0}x_{0}-M_{0}y_{0}\vert= 
M_{0}\vert \zeta x_{0}-y_{0}\vert= M_{0}\Vert \zeta x_{0}\Vert.
\end{equation}
Notice that in particular $x_{0}\leq x$ and
\begin{equation} \label{eq:doppel}
 \vert \zeta x_{0}-y_{0}\vert\leq \vert \zeta x-y\vert <(1/2)x^{-1}.
\end{equation}
Next we check \eqref{eq:leichter}. Assume \eqref{eq:leichter} is false.
Then there exist integers $x_{1},y_{1}$ with $x_{0}\neq x_{1}, x_{1}\leq x$ 
such that 
\begin{equation} \label{eq:trippel}
\vert \zeta x_{1}-y_{1}\vert<\vert \zeta x_{0}-y_{0}\vert
\leq \vert \zeta x-y\vert<(1/2)x^{-1}.  
\end{equation}
By an argument very similar to \eqref{eq:hilfsug} we can restrict to the case 
that $(x_{0},y_{0}),(x_{1},y_{1})$ are linearly independent.
A combination of \eqref{eq:doppel}, \eqref{eq:trippel} and the conditions $x_{0}\leq x, x_{1}\leq x$ 
and $x_{0}\neq x_{1}$ imply
\begin{equation} \label{eq:arith}
\left\vert \frac{y_{1}}{x_{1}}-\frac{y_{0}}{x_{0}}\right\vert 
\leq \left\vert \zeta-\frac{y_{1}}{x_{1}}\right\vert+\left\vert \zeta-\frac{y_{0}}{x_{0}}\right\vert
\leq \frac{1}{2xx_{1}}+\frac{1}{2xx_{0}} < \frac{1}{2x_{0}x_{1}}+\frac{1}{2x_{1}x_{0}}=\frac{1}{x_{0}x_{1}}.
\end{equation}
On the other hand, since $y_{0}/x_{0}\neq y_{1}/x_{1}$, we have
\begin{equation} \label{eq:geometric}
\left\vert \frac{y_{1}}{x_{1}}-\frac{y_{0}}{x_{0}}\right\vert
=\left\vert\frac{x_{0}y_{1}-x_{1}y_{0}}{x_{0}x_{1}}\right\vert \geq \frac{1}{x_{0}x_{1}}.
\end{equation}
The combination of \eqref{eq:arith} and \eqref{eq:geometric} contradicts the hypothesis 
and hence \eqref{eq:leichter} holds.
It remains to be shown that $x_{0}$ must be the {\em largest} convergent denominator not greater than $x$.
Otherwise, if $y^{\prime}/x^{\prime}$ is another convergent (in lowest terms) with $x_{0}<x^{\prime}\leq x$, 
then by the {\em monotonic} convergence of $\vert q_{n}\alpha-p_{n}\vert$
to $0$ for the convergents $p_{n}/q_{n}$ of $\alpha\in{\mathbb{R}}$, see~\cite[Satz~16]{perron}, 
the estimate \eqref{eq:trippel} would be satisfied for $x_{1}:=x^{\prime}, y_{1}:=y^{\prime}$, 
which we just falsified.
\end{proof}

\begin{definition}  \label{defi}
We call a positive integer $x$ a {\em best approximation for $\zeta$} if it satisfies
\[
\Vert \zeta x\Vert= \min_{1\leq v\leq x} \Vert \zeta v\Vert.   
\]
\end{definition}

The numbers $x_{0}$ from Lemma~\ref{help} are obviously best approximations.
We used in the proof of Lemma~\ref{help} that the denominators of convergents in lowest terms
are best approximations. It actually follows from
the law of best approximation~\cite[Satz~16]{perron} that the converse is true as well. 
Moreover, it is shown in~\cite[Satz~14]{perron} that 
at least one of two consecutive convergents satisfies \eqref{eq:eins}. It is further well-known that \eqref{eq:eins},
with the factor $1/2$ replaced by $1$, is valid for any convergent, see~\cite[Satz~10]{perron}.

The most technical ingredient in the proofs of the Theorems~\ref{spektrum},~\ref{theo}
is the following Lemma~\ref{lemma2}.

\begin{lemma} \label{lemma2}
Let $k$ be a positive integer and $\zeta$ be a real number.

Then there exists a constant $C=C(k,\zeta)>0$ such that for any integer $x>0$ the estimate 
\begin{equation} \label{eq:hain}
\max_{1\leq j\leq k}\Vert \zeta^{j} x\Vert < C\cdot x^{-1}
\end{equation}
implies $y/x=y_{0}/x_{0}$ for integers $(x_{0},y_{0})=1$ and $x_{0}^{k}$ divides $x$,
where $y$ denotes the closest integer to $\zeta x$.
A suitable choice for $C$ is given by $C=C_{0}:=(1/2)\cdot k^{-1}(1+\vert\zeta\vert)^{1-k}$.

Moreover, $y_{0}^{j}/x_{0}^{j}$ is a convergent of the continued fraction expansion of $\zeta^{j}$
for $1\leq j\leq k$.
Furthermore, if {\upshape(\ref{eq:hain})} holds for some pair $(x,C)=(Nx_{0}^{k},C)$
with an integer $N\geq 1$ and $C\leq C_{0}$, then
\begin{equation} \label{eq:haine}
\max_{1\leq j\leq k}\Vert \zeta^{j} x\Vert=N\cdot\max_{1\leq j\leq k}\Vert \zeta^{j} x_{0}^{k}\Vert.
\end{equation}
In particular, {\upshape(\ref{eq:hain})} holds for any pair $(x^{\prime},C)=(Mx_{0}^{k},C)$ 
with $1\leq M\leq N$ as well, and the minimum of the left hand side among those $x^{\prime}$ 
is obtained for $x^{\prime}=x_{0}^{k}$.
\end{lemma}

\begin{proof}
Suppose \eqref{eq:hain} holds for some $x$ and $C=C_{0}$. Denote by $y$ the closest integer to $\zeta x$
and let $y_{0}/x_{0}$ be the fraction $y/x$ in lowest terms.

Assumption \eqref{eq:hain} for $j=1$ leads to
\[
\left\vert \frac{y_{0}}{x_{0}}-\zeta\right\vert=\left\vert \frac{y}{x}-\zeta\right\vert < C_{0} x^{-2}.
\] 
Since $C_{0}<1/2<1$, we have $\vert y_{0}/x_{0}-\zeta\vert \leq 1$ and thus 
$\vert y_{0}/x_{0}\vert\leq 1+\vert\zeta\vert$. The combination of these facts yields for 
$1\leq j\leq k$ the estimate
\begin{equation}  \label{eq:tor}
\left\vert \frac{y_{0}^{j}}{x_{0}^{j}}-\zeta^{j}\right\vert
=\left\vert \frac{y_{0}}{x_{0}}-\zeta\right\vert\cdot  
\left\vert \left(\frac{y_{0}}{x_{0}}\right)^{j-1}+\cdots+\zeta^{j-1}\right\vert
< C_{0}x^{-2}\cdot k\left(1+\vert\zeta\vert\right)^{k-1}=\frac{1}{2}x^{-2}.
\end{equation} 
Suppose $x_{0}^{k}\nmid x$. Then, since $x_{0}\vert x$, the integer $x$ has a representation in base $x_{0}$ as 
\[
x=b_{1}x_{0}+b_{2}x_{0}^{2}+\cdots+b_{k-1}x_{0}^{k-1}+b_{k}x_{0}^{k}+\cdots+b_{l}x_{0}^{l},
\]
where at least one of $\{b_{1},b_{2},\ldots,b_{k-1}\}$ is not zero.
Put $u=i+1\in{\{2,3,\ldots,k\}}$ with $i$ the smallest index such that $b_{i}\neq 0$.
By construction, for all $j\neq i$ we have $b_{j}x_{0}^{j}(y_{0}^{u}/x_{0}^{u})\in{\mathbb{Z}}$.
Hence, using $(x_{0},y_{0})=1$ and $b_{i}\neq 0$, we have the estimate
\begin{equation} \label{eq:wahn}
\left\Vert x\frac{y_{0}^{u}}{x_{0}^{u}}\right\Vert= \left\Vert b_{i}x_{0}^{i} \frac{y_{0}^{u}}{x_{0}^{u}} \right\Vert
=\left\Vert \frac{b_{i}y_{0}^{u}}{x_{0}}\right\Vert\geq x_{0}^{-1}.
\end{equation}
On the other hand, the estimate \eqref{eq:tor} for $j=u$ implies
\begin{equation} \label{eq:sinn}
\left\vert x\left(\zeta^{u}-\frac{y_{0}^{u}}{x_{0}^{u}}\right)\right\vert
\leq \frac{1}{2} x^{-1}\leq \frac{1}{2}\cdot x_{0}^{-1}.
\end{equation}
The combination of \eqref{eq:wahn} and \eqref{eq:sinn} and triangular inequality imply
\[
\max_{1\leq j\leq k}\Vert \zeta^{j} x\Vert\geq \Vert \zeta^{u} x\Vert> \frac{1}{2}x_{0}^{-1}\geq \frac{1}{2}x^{-1},
\]
contradicting \eqref{eq:hain} since $C_{0}<1/2$. Hence, $x_{0}^{k}\vert x$.

From $x_{0}^{k}\vert x$ we infer $x_{0}^{k}\leq x$, and \eqref{eq:tor} yields
\[
\left\vert \frac{y_{0}^{j}}{x_{0}^{j}}-\zeta^{j}\right\vert
< \frac{1}{2}x^{-2}\leq \frac{1}{2}x_{0}^{-2k}=
\frac{1}{2}x_{0}^{2j-2k}\cdot (x_{0}^{j})^{-2}\leq \frac{1}{2}\cdot (x_{0}^{j})^{-2}, \qquad 1\leq j\leq k.
\]
Since clearly $(x_{0}^{j},y_{0}^{j})=1$ for any $1\leq j\leq k$, Lemma~\ref{help}
implies that $y_{0}^{j}/x_{0}^{j}$ is indeed a convergent of $\zeta^{j}$ for every $1\leq j\leq k$.
Finally, we show \eqref{eq:haine}. Let $x=Nx_{0}^{k}$. It follows from \eqref{eq:tor} 
that $y_{j}:=Nx_{0}^{k-j}y_{0}^{j}$ is the closest integer to $\zeta^{j}x$, for $1\leq j\leq k$.
Thus, we obtain 
\[
\Vert \zeta^{j}x\Vert=\vert \zeta^{j}x-y_{j}\vert
=N\cdot\vert \zeta^{j}x_{0}^{k}-x_{0}^{k-j}y_{0}^{j}\vert=N\cdot\Vert \zeta^{j}x_{0}^{k}\Vert, \qquad 1\leq j\leq k.
\]
Hence the relation holds for the maximum as well.
\end{proof}

\begin{remark}
Analyzing the estimates in \eqref{eq:tor}, the constant $C_{0}$ can be improved
if we additionally assume $x$ to be sufficiently large. One finds that
Lemma~\ref{lemma2} actually holds with $C_{0}=(1/2)\cdot L_{k}(\zeta)^{-1}-\epsilon$ 
for arbitrary small $\epsilon>0$, where $L_{k}(\zeta):=\max_{1\leq j\leq k} (j\vert\zeta\vert^{j-1})$,
for all $x\geq x_{0}(\epsilon)$. 
Moreover, in case of $\vert \zeta\vert<1/2$, the maximum in $L_{k}(\zeta)$
is obtained for $j=1$ and we may put $\epsilon=0$, which yields $C_{0}=1/2$. 
\end{remark}

\begin{remark}
It is not hard to see Lemma~\ref{lemma2} would be wrong with right hand side in \eqref{eq:hain} replaced
by $(1/2)\cdot x_{0}^{-1}$ for any $\zeta$ with $\widehat{\lambda}_{k}(\zeta)>1/k$. In particular,
for $k=2$ and extremal numbers $\zeta$ mentioned in Section~\ref{intro}. 
The proof of the false stronger version fails since \eqref{eq:sinn} is no longer correct.
\end{remark}

\section{Proof of Theorems \ref{spektrum}, \ref{theo}}  \label{beweis}

First we prove Theorem~\ref{theo} with a method very similar to to the proof 
of Lemma~\ref{lemma2}. It might be possible to deduce Theorem~\ref{theo} directly
from this lemma, however, due to technical difficulties, we prefer to prove it directly.

\begin{proof}[Proof of Theorem~\ref{theo}]
Consider $k,\zeta$ fixed. The assertion is trivial for $\zeta\in{\mathbb{Q}}$ 
and in case of $\lambda_{1}(\zeta)=1$, so we may assume $\zeta$ is irrational and $\lambda_{1}(\zeta)>1$.   

Let $1<T<\lambda_{1}(\zeta)$ be arbitrary. By definition of the quantity $\lambda_{1}(\zeta)$,
and since $\zeta\notin{\mathbb{Q}}$,
there exist arbitrarily large coprime $x_{0},y_{0}$ with the property that 
\[
\left\vert\zeta-\frac{y_{0}}{x_{0}}\right\vert\leq x_{0}^{-T-1}.
\]
For sufficiently large $x_{0}$ and a constant $D_{0}=D_{0}(k,\zeta)$, 
similarly as in \eqref{eq:tor} we deduce 
\begin{equation}  \label{eq:apuz}
\left\vert\zeta^{j}-\frac{y_{0}^{j}}{x_{0}^{j}}\right\vert<D_{0}x_{0}^{-T-1},\qquad 1\leq j\leq k.
\end{equation}
We distinguish the cases $\lambda_{1}(\zeta)>k$ and $\lambda_{1}(\zeta)\leq k$. 

Case 1: $\lambda_{1}(\zeta)> k$. Then we may assume $T>k$ as well. 
Let $X:=x_{0}^{k}/2$. Write $1\leq x\leq X$ in base $x_{0}$ as 
\[
x=b_{0}+b_{1}x_{0}+b_{2}x_{0}^{2}+\cdots+b_{k-1}x_{0}^{k-1}+b_{k}x_{0}^{k}+\cdots+b_{l}x_{0}^{l},
\]
and put $u=i+1\in{\{1,2,\ldots,k\}}$ with $i$ the smallest index such that $b_{i}\neq 0$. 
Since $x_{0},y_{0}$ are coprime and $b_{i}\neq 0$, we have
\begin{equation} \label{eq:sheyn}
\left\Vert x\frac{y_{0}^{u}}{x_{0}^{u}}\right\Vert = 
\left\Vert b_{i}x_{0}^{u-1}\frac{y_{0}^{u}}{x_{0}^{u}}\right\Vert=
\left\Vert\frac{b_{i}y_{0}^{u}}{x_{0}}\right\Vert \geq \frac{1}{x_{0}}.
\end{equation}
Moreover, \eqref{eq:apuz} yields for $1\leq x\leq X$ the upper bounds
\begin{equation} \label{eq:appoi}
\left\vert x\left(\zeta^{u}-\frac{y_{0}^{u}}{x_{0}^{u}}\right)\right\vert 
\leq \frac{X}{2} \left\vert \zeta^{u}-\frac{y_{0}^{u}}{x_{0}^{u}} \right\vert 
\leq \frac{D_{0}}{2}x_{0}^{k-T-1}.
\end{equation}
Since $T>k$, the right hand side is smaller than $(1/2)x_{0}^{-1}$ for large $x_{0}$,
so combining \eqref{eq:sheyn}, \eqref{eq:appoi} with triangular inequality yields for $1\leq x\leq X$ the estimate
\[
M_{x}(\zeta):=\max_{1\leq j\leq k} \Vert \zeta^{j}x\Vert\geq \Vert \zeta^{u}x\Vert\geq \frac{1}{2}x_{0}^{-1}.
\]
Using the definition of $\widehat{\lambda}_{k}$ we conclude
\[
\widehat{\lambda}_{k}(\zeta)\leq \liminf_{X\to\infty}\max_{1\leq x\leq X} 
-\frac{\log M_{x}(\zeta)}{\log X}\leq \frac{1+\frac{\log 2}{\log x_{0}}}{k-\frac{\log 2}{\log x_{0}}},
\]
and with $X\to\infty$ or equivalently $x_{0}\to\infty$ indeed $\widehat{\lambda}_{k}(\zeta)\leq 1/k$.

Case 2: $\lambda_{1}(\zeta)\leq k$. Define $X:=(1/2)D_{0}^{-1}x_{0}^{T}$, 
and again write $x\leq X$ in base $x_{0}$ and define $i,u,b_{.}$ as in case 1.
We have $0\leq i\leq \lfloor T\rfloor\leq T$, such that from $T<k$ we infer that $1\leq u\leq k$.
We have \eqref{eq:sheyn} precisely as in case 1, such as
\begin{equation} \label{eq:sheyn3}
\left\vert x\left(\zeta^{u}-\frac{y_{0}^{u}}{x_{0}^{u}}\right)\right\vert
\leq X\left\vert \zeta^{u}-\frac{y_{0}^{u}}{x_{0}^{u}} \right\vert
\leq \frac{1}{2D_{0}}x_{0}^{T}\cdot D_{0}x_{0}^{-T-1}=\frac{1}{2}x_{0}^{-1}.
\end{equation}
So combining \eqref{eq:sheyn} and \eqref{eq:sheyn3} and triangular inequality yields
\[
M_{x}(\zeta):=\max_{1\leq j\leq k} \Vert \zeta^{j}x\Vert\geq \Vert \zeta^{u}x\Vert\geq \frac{1}{2}x_{0}^{-1}.
\]
Again we conclude
\[
\widehat{\lambda}_{k}(\zeta)\leq \liminf_{X\to\infty}\max_{1\leq x\leq X} 
-\frac{\log M_{x}(\zeta)}{\log X}
\leq \frac{1+\frac{\log 2}{\log x_{0}}}{T-\frac{\log D_{0}+\log 2}{\log x_{0}}}.
\]
As we may choose $T$ arbitrarily close to $\lambda_{1}(\zeta)$, indeed
$\widehat{\lambda}_{k}(\zeta)\leq 1/\lambda_{1}(\zeta)$ follows again with 
$X\to\infty$ or equivalently $x_{0}\to\infty$.
\end{proof}

Next we prove Theorem~\ref{spektrum} using Lemma~\ref{lemma2} and Lemma~\ref{help}.

\begin{proof}[Proof of Theorem~\ref{spektrum}]
In view of \eqref{eq:holds}, for the assertion on $\lambda_{k}(\zeta)$
we only have to show that provided that $\lambda_{k}(\zeta)>1$ holds, we have 
\begin{equation} \label{eq:dieletzte}
\lambda_{k}(\zeta)\leq \frac{\lambda_{1}(\zeta)-k+1}{k}.
\end{equation}
The definition of the quantity $\lambda_{k}(\zeta)$ implies that for any fixed 
$1<T<\lambda_{k}(\zeta)$, the inequality
\begin{equation} \label{eq:xy}
\max_{1\leq j\leq k} \Vert\zeta^{j} x\Vert \leq x^{-T}
\end{equation}
has arbitrarily large integer solutions $x$. One checks that for any $\tau>0$ and sufficiently large 
$x>~\hat{x}(\tau,T):=\tau^{1/(1-T)}$ we have $x^{-T}<\tau x^{-1}$. 
Choosing $\tau\leq C_{0}$ with $C_{0}<1/2$ from Lemma~\ref{lemma2}, condition \eqref{eq:xy} ensures 
we may apply both Lemma~\ref{lemma2} and
Lemma~\ref{help} for $x\geq \hat{x}$, with coinciding pairs $x_{0},y_{0}$ such that 
$y_{0}/x_{0}$ is the reduced fraction $y/x$. 
Further let $M_{0}$ be as in Lemma~\ref{help}. Writing $M_{0}=x_{0}^{\eta}$, by Lemma~\ref{lemma2}  
we infer $\eta \geq k-1$. Moreover, define $T_{0}$ implicitly by 
$x_{0}^{-T_{0}}=\vert \zeta x_{0}-y_{0}\vert$, i.e.
\[
T_{0}=-\frac{\log \vert \zeta x_{0}-y_{0}\vert}{\log x_{0}}.
\]
The derived properties yield
\[
T\leq -\frac{\log \Vert \zeta x\Vert}{\log x}= 
-\frac{\log(M_{0}\vert \zeta x_{0}-y_{0}\vert)}{\log(M_{0}x_{0})}
\leq \frac{T_{0}-\eta}{1+\eta}
\leq \frac{T_{0}-(k-1)}{1+(k-1)}=\frac{T_{0}-k+1}{k}.
\]
Since this is true for arbitrarily large values of $x$ (and thus $x_{0}$)
and we may choose $T$ arbitrarily close to $\lambda_{k}(\zeta)$, the definition
of $T_{0}$ implies \eqref{eq:dieletzte}. 

Since $\lambda_{1}(\zeta)=k\lambda_{k}(\zeta)+k-1>2k-1\geq k$, 
the assertion on $\widehat{\lambda}_{j}(\zeta)$ follows from Theorem~\ref{theo}.
\end{proof}

We actually proved something stronger than Theorem~\ref{spektrum}.
We point out the more general results evolved from the proof as a corollary.

\begin{corollary}  \label{korollar}
Let $k\geq 2$ be an integer and $\zeta$ be a real number. For any fixed $T>1$,
there exists $\hat{x}=\hat{x}(T,\zeta)$, such that the estimate
\[
\max_{1\leq j\leq k}\Vert \zeta^{j} x\Vert \leq x^{-T}
\]
for an integer $x\geq \hat{x}$ implies the
existence of $x_{0},y_{0},M_{0}$ as in Lemma~{\upshape\ref{help}} with the properties
\begin{equation}  \label{eq:allesgilt}
x\geq x_{0}^{k}, \qquad M_{0}\geq x_{0}^{k-1}, \qquad \vert \zeta x_{0}-y_{0}\vert\leq x_{0}^{-kT-k+1}.
\end{equation}
Similarly, if for $C_{0}=C_{0}(k,\zeta)$ from Lemma~{\upshape\ref{lemma2}} the inequality
\[
\max_{1\leq j\leq k}\Vert \zeta^{j} x\Vert < C_{0}\cdot x^{-1}
\]
has an integer solution $x>0$, then {\upshape(\ref{eq:allesgilt})} holds with $T=1$. 
\end{corollary} 

For direct consequences of Corollary~\ref{korollar}, see Section~\ref{vier}.

\section{Diophantine approximation in fractal sets}  \label{cantorset}

\subsection{Definitions and results} \label{deandre}
The middle-third Cantor set $\mathscr{C}$
is defined as the real numbers $a$ in $[0,1]$ that can be written in the form
\[
a=c_{1}3^{-1}+c_{2}3^{-2}+\cdots, \qquad \qquad c_{i}\in{\{0,2\}}.
\]  
The spectrum of the quantity $\lambda_{k}(\zeta)$ with the restriction 
that $\zeta$ belongs to the Cantor set has been studied. For $k=1$, the question is solved
by the following constructive result~\cite[Theorem~2]{buge}. 
We use a slightly different notation than the one in~\cite{buge}
for correlation with our upcoming results. 

\begin{theorem}[Bugeaud]  \label{bugerg}
Let $\tau\in{[1,\infty)}$ and $\alpha>0$. Any number
\begin{equation}  \label{eq:anna}
\zeta=2\sum_{n\geq 1} 3^{-\lceil\alpha(1+\tau)^{n}\rceil}
\end{equation}
belongs to $\mathscr{C}$ and satisfies $\lambda_{1}(\zeta)=\tau$. In particular, the spectrum
of $\lambda_{1}$ restricted to $\mathscr{C}$ equals $[1,\infty]$. 
\end{theorem}

Indeed, the case $\lambda_{1}(\zeta)=\infty$ not explicitly mentioned in~\cite[Theorem~2]{buge} 
is obtained similarly by a sequence with hyper-exponential growth in the exponent, 
such that we can include the value $\infty$ in Theorem~\ref{bugerg}.

The best current result concerning the spectrum of $\lambda_{k}$ within $\mathscr{C}$
for $k\geq 2$, which is~\cite[Theorem~BL]{bug}, originates in $\zeta$ as in \eqref{eq:anna}
incorporating~\cite[Theorem~7.7]{bugg}.
\begin{theorem}[Bugeaud, Laurent] \label{alteschranke}
Let $k\geq 2$ be an integer. The spectrum of $\lambda_{k}$ among $\zeta$ in $\mathscr{C}$ contains 
the interval $[(1+\sqrt{4k^{2}+1})/(2k),\infty]$.
\end{theorem}
Again, as conjectured for real numbers $\zeta$ in Problem~\ref{problem1}, there is reason 
to believe that the spectrum actually equals $[1/k,\infty]$.
The following immediate consequence of Theorem~\ref{spektrum}
yields an improvement of Theorem~\ref{alteschranke}.

\begin{theorem} \label{schwach}
Let $k\geq 2$ be an integer and $\alpha>0, \rho>0$. Then
$\zeta=2\sum_{n\geq 1} 3^{-\lceil\alpha (k(1+\rho))^{n}\rceil}$
belongs to the Cantor set. 
If $\rho\in{(0,1]}$, then
\begin{equation} \label{eq:warm}
\max\left\{\rho,\frac{1}{k}\right\}\leq \lambda_{k}(\zeta)\leq 1.
\end{equation}
If $\rho\in{(1,\infty]}$, we have equality
\begin{equation}  \label{eq:wbein}
\lambda_{k}(\zeta)=\rho.
\end{equation}
         
\end{theorem}

\begin{proof}
By Theorem~\ref{bugerg} with $\tau:=k\rho+k-1>1$, we have $\lambda_{1}(\zeta)=k\rho+k-1$. 
If $\rho>1$, then $k< k_{0}:=\lceil (\lambda_{1}(\zeta)+1)/2\rceil$ such that with Corollary~\ref{bestens},
we obtain \eqref{eq:wbein}. If $\rho\in{(0,1]}$, then $k\geq k_{0}$ and the 
assertion \eqref{eq:warm} again follows from Corollary~\ref{bestens}. 
\end{proof}

Theorem~\ref{schwach} obviously yields the improvement
of Theorem~\ref{alteschranke} so that the spectrum of $\lambda_{k}$ contains $[1,\infty]$.
However, we want to prove the more general statement Theorem~\ref{meinsatz}. It 
is more flexible in the choice of suitable $\zeta$ and extends Theorem~\ref{schwach} 
to expansions in an arbitrary base.

\begin{theorem}  \label{meinsatz}
Let $k\geq 2, b\geq 2$ be integers and $\rho\in{(0,\infty]}$. 
Let $(a_{n})_{n\geq 1}$ be a strictly increasing sequence of positive integers with the property
\begin{equation} \label{eq:folge}
\lim_{n\to\infty} \frac{a_{n+1}}{a_{n}}=k(\rho+1).
\end{equation}
Let
\begin{equation} \label{eq:zeter}
\zeta=\sum_{n\geq 1} b^{-a_{n}}.
\end{equation}
If $\rho\in{(0,1)}$, then
\begin{equation} \label{eq:arm}
\max\left\{\frac{1}{k},\rho\right\}\leq \lambda_{k}(\zeta)\leq 1.
\end{equation}
If $\rho\in{[1,\infty]}$, we have equality
\begin{equation}  \label{eq:bein}
\lambda_{k}(\zeta)=\rho.
\end{equation}                                                               
Moreover, if $\rho\geq 1/k$, then $\widehat{\lambda}_{k}(\zeta)=1/k$.
\end{theorem}

Note the similarity to the constructions of Theorem~\ref{bugerg}, where
the analogue result was established for $k=1$ and the sequence  
$a_{n}=\lceil \alpha(k(1+\rho))^{n}\rceil$ for $\alpha>0$. 
Roughly speaking, the additional factor $k$ in the quotient $a_{n+1}/a_{n}$ allows for generalizing 
the one-dimensional result. However, the methods of the proofs of Theorem~\ref{bugerg}
and Theorem~\ref{meinsatz} are much different. The approach in this paper is rather
connected to the one in~\cite{sa}, where a slightly weaker result than Theorem~\ref{bugerg} was established.

We encourage the reader to compare the following corollary to 
Theorem~\ref{meinsatz} with~\cite[Theorem~1]{buge}, which we will not state,
where a more general result in the special case $k=1$ was established.

\begin{corollary}  \label{verallg}
Let $k\geq 2$, $b\geq 2$ be integers and $\mathscr{A}\subset \{0,1,\ldots,b-1\}$
of cardinality $\vert \mathscr{A}\vert\geq 2$. 
The spectrum of the approximation constant $\lambda_{k}(\zeta)$, restricted to
$\zeta\in{(0,1)}$ whose expansion in base $b$ have all digits in $\mathscr{A}$, contains $[1,\infty]$.
In particular, for any $\epsilon>0$ there exists a set $\mathscr{B}$ of Hausdorff dimension less than $\epsilon$
such that the spectrum of $\lambda_{k}$ within $\mathscr{B}$ contains $[1,\infty]$.
\end{corollary}

For the proof of Theorem~\ref{meinsatz} we will need an estimate for the 
concrete numbers $\zeta$ in \eqref{eq:zeter}. This will be established in Lemma~\ref{danngenug}.
For its proof we apply a proposition connected to Lemma~\ref{help}.  

\begin{proposition} \label{cor}
Let $\zeta\in{\mathbb{R}}$. Then for no parameter $Q>0$ the system
\begin{equation} \label{eq:lindep}
\vert m\vert \leq Q, \qquad  \vert \zeta m- n\vert< \frac{1}{2Q}
\end{equation}
has two linearly independent solutions $(m,n)\in{\mathbb{Z}^{2}}$.
\end{proposition}

\begin{proof}
We may assume $Q\geq 1$, otherwise there do not exist two linearly independent vectors anyway.
Hence \eqref{eq:lindep} implies $m\neq 0$, so $m>0$ is no restriction (for else if $m<0$ consider $(-m,-n)$).
For fixed $Q$, say $(m_{1},n_{1})\in{\mathbb{Z}_{>0}\times \mathbb{Z}}$ 
is a solution to \eqref{eq:lindep} with largest $m_{1}$ among all such solutions. 
We have to show that any vector $(m_{2},n_{2})\in{\mathbb{Z}_{>0}\times \mathbb{Z}}$ 
linearly independent to $(m_{1},n_{1})$ with 
$m_{2}\leq m_{1}\leq Q$ satisfies $\vert \zeta m_{2}- n_{2}\vert\geq (1/2)Q^{-1}$. 
We infer 
\[
\left\vert \zeta-\frac{n_{1}}{m_{1}}\right\vert <\frac{1}{2m_{1}Q}\leq \frac{1}{2m_{1}m_{2}}
\]
from \eqref{eq:lindep}. Thus, the linear independence condition implies
\[
\left\vert \zeta-\frac{n_{2}}{m_{2}}\right\vert \geq 
\left\vert \frac{n_{1}}{m_{1}}-\frac{n_{2}}{m_{2}}\right\vert-\left\vert \zeta-\frac{n_{1}}{m_{1}}\right\vert
\geq \frac{1}{m_{1}m_{2}}-\frac{1}{2m_{1}m_{2}}=\frac{1}{2m_{1}m_{2}}\geq \frac{1}{2m_{2}Q}.
\]
Multiplying with $m_{2}$ yields the assertion.
\end{proof}

An alternative proof of Proposition~\ref{cor} is obtained by regarding it as a special case of  
Minkowski's second lattice point theorem~\cite{minkowski} on convex bodies, in the plane.

\begin{lemma}  \label{danngenug}
Let $k\geq 2, b\geq 2$ be integers, $\rho>0$ and $\zeta$ be as in {\upshape(\ref{eq:zeter})}
for some sequence $(a_{n})_{n\geq 1}$ as in {\upshape(\ref{eq:folge})}. 
Then for $(x,y)\in{\mathbb{Z}^{2}}$ with sufficiently large $x$, the estimate
\begin{equation} \label{eq:harz}
\vert\zeta x-y\vert\leq x^{-\frac{k}{k-1}}
\end{equation}
implies $(x,y)$ an integral multiple of some
\begin{equation} \label{eq:vektoren}
\underline{x}_{n}:=(x_{n},y_{n}):=(b^{a_{n}},\sum_{i\leq n} b^{a_{n}-a_{i}}).
\end{equation}
\end{lemma}

\begin{proof}
First note that for any fixed $n$, the entries $x_{n},y_{n}$ of
the vectors $\underline{x}_{n}=(x_{n},y_{n})$ are coprime, 
since $x_{n}$ consists of prime factors dividing $b$ and $y_{n}\equiv 1\bmod b$.
Hence any vector $(x,y)\in{\mathbb{Z}^{2}}$ which is no integral multiple of 
$\underline{x}_{n}$ is actually linearly independent from it.

Assume the lemma is false. Then by the above observation,
there exist arbitrarily large $(x,y)$ for which \eqref{eq:harz} holds
and which are linearly independent to all $\underline{x}_{n}$. 
Let $\delta>0$ not be too large, in particular $\delta=1$ will be a proper choice if we
assume $x$ (or $n$) is sufficiently large. 
Say $n$ is the index with $b^{a_{n}}\leq x<b^{a_{n+1}}$. 

First suppose $x\leq b^{a_{n+1}-(1+\delta)a_{n}}$. Put $Q=x$.
If $n$ or equivalently $x$ is sufficiently large, then by assumption \eqref{eq:harz}
we have
\[
-\frac{\log \vert \zeta x-y\vert}{\log Q}
= -\frac{\log \vert \zeta x-y\vert}{\log x}
\geq \frac{k}{k-1}> 1.
\]
On the other hand the estimate
\begin{equation} \label{eq:trick}
\Vert b^{a_{n}}\zeta\Vert = \sum_{i\geq n+1} b^{a_{n}-a_{i}}
\leq 2\cdot b^{a_{n}-a_{n+1}}
\end{equation}
implies
\begin{equation} \label{eq:fried}
-\frac{\log \vert \zeta x_{n}-y_{n}\vert}{\log Q}= -\frac{\log \Vert b^{a_{n}}\zeta\Vert}{\log x}
\geq \frac{a_{n+1}-a_{n}-\frac{\log 2}{\log b}}{a_{n+1}-(1+\delta)a_{n}}>1.
\end{equation}
Hence, for some fixed $\epsilon>0$ and arbitrarily large $n$, the system
\begin{equation} \label{eq:wievorher}
\vert M\vert \leq Q, \qquad \vert \zeta M-N\vert \leq Q^{-1-\epsilon}
\end{equation}
has two linearly independent integral solutions $(M,N)=(x,y),(M,N)=(x_{n},y_{n})$. 
Since the above holds for all $n\geq 1$ and we may assume $Q=x>2^{1/\epsilon}$,
we infer a contradiction to Proposition~\ref{cor}.
 
In the remaining case $b^{a_{n+1}-(1+\delta)a_{n}}\leq x<b^{a_{n+1}}$, put $Q=b^{a_{n+1}}$.
For sufficiently large $n$, clearly \eqref{eq:trick} for $n$ replaced by $n+1$ shows that
$(x_{n+1},y_{n+1})$ satisfies \eqref{eq:wievorher}
for some $\epsilon>0$ (actually any $\epsilon\geq k(\rho+1)-1\geq k-1\geq 1$).
On the other hand, \eqref{eq:harz} yields
\begin{equation} \label{eq:frie}
-\frac{\log \vert \zeta x-y\vert}{\log Q}= -\frac{\log \vert \zeta x-y\vert}{\log x}\cdot\frac{\log x}{\log Q}
\geq \frac{k}{k-1}\cdot \frac{a_{n+1}-(1+\delta)a_{n}}{a_{n+1}}.
\end{equation}
For $\rho>0$, we have
\[
\lim_{n\to\infty, \delta\to 0}\frac{a_{n+1}-(1+\delta)a_{n}}{a_{n+1}}= \frac{k(\rho+1)-1}{k(\rho+1)}> \frac{k-1}{k}.
\]
Hence, the right hand side in \eqref{eq:frie} is strictly greater than $1$ and 
consequently the left is too. Thus
for some $\epsilon>0$ the system \eqref{eq:wievorher} has linearly independent integral solutions 
$(x,y),(x_{n+1},y_{n+1})$ again, contradiction to Proposition~\ref{cor} for large $n$ (resp. $Q$).
\end{proof}

\begin{remark}  \label{remarkch}
The continued fraction expansion of numbers $\zeta$ as in Theorem~\ref{meinsatz} can be explicitly
established using some variant of the Folding~Lemma,
see~\cite{buge} or~\cite{fra}. This should allow for 
proving the assertion of Lemma~\ref{danngenug} even for slightly larger exponents than $-k/(k-1)$ in \eqref{eq:harz}. 
It is reasonable that even the optimal value in \eqref{eq:harz}
in the dependence of $\rho$ can be determined
for which the claim of Lemma~\ref{danngenug} holds. 
It is possible to show, though, that Lemma~\ref{danngenug} does not apply with 
exponent $-1-\epsilon$ for some $\epsilon>0$ (otherwise the proof of Theorem~\ref{meinsatz}
could be simplified).
However, improvements of this kind are not necessary for our purposes. 
In fact, we only need the much weaker bound $-2k+1$ instead 
of $-k/(k-1)$ for the proof of Theorem~\ref{meinsatz}.
\end{remark}

\begin{remark}
In fact, Lemma~\ref{danngenug} is still true by essentially the same proof if 
we relax the assumption \eqref{eq:folge} to the weaker condition $\liminf_{n\to\infty} a_{n+1}/a_{n}\geq k(\rho+1)$. 
\end{remark}

\subsection{Proof of Theorem~\ref{meinsatz}}
We now prove Theorem~\ref{meinsatz} by using Lemma~\ref{help}, Lemma~\ref{danngenug} 
and Corollary~\ref{korollar}.
Lower bounds for $\lambda_{k}(\zeta)$ in Theorem~\ref{meinsatz} will be rather straightforward to derive
by looking at integers of the form $x=b^{ka_{n}}$ for large $n$,
whereas the proof of more interesting upper bounds is slightly technical. 
We sketch the outline of the proof of the upper bounds. We distinguish
between integers $x$ with the property that $\Vert \zeta x_{0}\Vert<x_{0}^{-2k+1}$ 
for $x_{0}$ the largest best approximation $\leq x$, see Definition~\ref{defi}, and those for
which this inequality is wrong. Lemma~\ref{help} and Lemma~\ref{danngenug} 
allow for an easy classification
of the values $x$ belonging to the first class to which
Lemma~\ref{lemma2} and Lemma~\ref{help} can be effectively applied to obtain upper bounds. For the 
remaining class of integers $x$, the negated formulation of Corollary~\ref{korollar} immediately
yields the upper bound $1$.   
 
\begin{proof}[Proof of Theorem~\ref{meinsatz}]
We consider $k\geq 2$ and $\rho>0$ fixed, and a corresponding sequence $(a_{n})_{n\geq 1}$
and $\zeta$ as in \eqref{eq:zeter} is constructed via the sequence.
Note that
\begin{equation} \label{eq:est}
\Vert b^{a_{n}}\zeta\Vert = \sum_{i\geq n+1} b^{a_{n}-a_{i}}\leq 2\cdot b^{a_{n}-a_{n+1}}.
\end{equation}
We first prove the assertion on $\widehat{\lambda}_{k}$.
Assuming $\rho\geq 1/k$, for any $\delta>0$ and sufficiently large $n\geq\widehat{n}(\delta)$ we have
\[
a_{n+1}\geq (k+k\rho-\delta) a_{n}\geq (k+1-\delta)a_{n}.
\]
If we choose integers $x$ of the form $b^{a_{n}}$, the estimate \eqref{eq:est} and $\delta\to 0$ imply 
\[
\lambda_{1}(\zeta)\geq \limsup_{n\geq 1} -\frac{\log(2\cdot b^{a_{n}-a_{n+1}})}{\log b^{a_{n}}}
\geq \limsup_{n\geq 1} \frac{(k+1)a_{n}-a_{n}}{a_{n}}= k.
\]
For any such real number $\zeta$, the assertion follows directly from Theorem~\ref{theo}.     
						
To prove \eqref{eq:arm} and \eqref{eq:bein}, we show
\begin{equation}  \label{eq:kopf}
\max\left\{\frac{1}{k},\rho\right\}\leq \lambda_{k}(\zeta)\leq \max\left\{1,\rho\right\}.
\end{equation}																				
We start with the left inequality. 
We only have to show $\lambda_{k}(\zeta)\geq \rho$, the other inequality $\lambda_{k}(\zeta)\geq 1/k$ is
trivial by \eqref{eq:dirichlet}. It suffices to consider integers $x$ of the form 
$x=b^{ka_{n}}$. Write $\zeta=S_{n}+\epsilon_{n}$ with
\[
S_{n}=\sum_{i=1}^{n} b^{-a_{i}}, \qquad \epsilon_{n}=\sum_{i=n+1}^{\infty} b^{-a_{i}}.
\]
Since $S_{n}<1, \epsilon_{n}<1$ and the binomial coefficients are bounded above by $k!$ , we have that
\[
\zeta^{j}=\sum_{i=0}^{j} \binom{j}{i}S_{n}^{i}\epsilon_{n}^{j-i}=S_{n}^{j}+O(\epsilon_{n}), \qquad 1\leq j\leq k, 
\]
as $n\to \infty$, with the implied constant depending on $k$ only. The crucial point 
now is that $xS_{n}^{j}$ is an integer
for $1\leq j\leq k$ by construction. Thus for some constant $C_{0}>0$ 
independent of $n$ and $C_{1}=2C_{0}$, we have
\begin{equation} \label{eq:tuti}
\left\Vert x\zeta^{j}\right\Vert\leq C_{0} \sum_{i=n+1}^{\infty} b^{-a_{i}}x
=C_{0}\cdot b^{ka_{n}}\sum_{i=n+1}^{\infty} b^{-a_{i}}\leq C_{1}\cdot b^{ka_{n}-a_{n+1}}, \qquad 1\leq j\leq k.
\end{equation}
The condition \eqref{eq:folge} implies for any $\nu>0$ and sufficiently large $n\geq \hat{n}(\nu)$
\[
ka_{n}-a_{n+1}\leq (-k\rho+\nu) a_{n}.
\]
For sufficiently large $n$ (or equivalently $x$), combination with \eqref{eq:tuti} yields
\[
\max_{1\leq j\leq k}\left\Vert x\zeta^{j}\right\Vert \leq C_{1}\cdot b^{(-k\rho+\nu) a_{n}}=
C_{1}\cdot x^{-\rho+\frac{\nu}{k}}\leq x^{-\rho+\frac{2\nu}{k}}.
\]
Since $\nu$ can be taken arbitrarily small we indeed obtain $\lambda_{k}(\zeta)\geq \rho$ for any fixed $\rho>0$.

We are left to prove the right hand side of \eqref{eq:kopf}, which we do
indirectly. Suppose there exists $\rho>0$ and $\zeta$ as in Theorem~\ref{meinsatz} 
such that $\lambda_{k}(\zeta)>\max\{1,\rho\}$. Then for $\epsilon=(\lambda_{k}(\zeta)+1)/2-1>0$, 
the inequality
\[
\max_{1\leq j\leq k}\Vert \zeta^{j} x\Vert\leq x^{-1-\epsilon}
\]
has arbitrarily large solutions $x$. Consequently Corollary~\ref{korollar} applies.
It yields that $x=M_{0}x_{0}$ for some best approximation $x_{0}$ such that 
$\vert \zeta x_{0}-y_{0}\vert< x_{0}^{-2k+1}$ for some $y_{0}$ with $(x_{0},y_{0})=1$, and $M_{0}\geq x_{0}^{k-1}$.
Note that $k/(k-1)< 2k-1$ for $k\geq 2$. Recall from the proof of Lemma~\ref{danngenug} that
for any fixed $n$, the entries of the vectors $\underline{x}_{n}=(x_{n},y_{n})$ defined in \eqref{eq:vektoren} 
are coprime too. Thus Lemma~\ref{danngenug} shows that for large $x_{0}$, the inequality 
$\vert \zeta x_{0}-y_{0}\vert< x_{0}^{-2k+1}$ can be satisfied only if 
$(x_{0},y_{0})=\underline{x}_{n}$ for some $n$.
Hence we can write $x_{0}=b^{a_{n}}$, and consequently $M_{0}\geq b^{(k-1)a_{n}}$ and $x\geq b^{ka_{n}}$. 
Observe that
\[
\Vert b^{a_{n}}\zeta\Vert = \sum_{i\geq n+1} b^{a_{n}-a_{i}}\geq b^{a_{n}-a_{n+1}}.
\]
Hence, \eqref{eq:multiplizieren} of Lemma~\ref{help} yields
\begin{equation} \label{eq:huhn}
\Vert \zeta x\Vert=M_{0}\vert \zeta x_{n}-y_{n}\vert\geq 
b^{(k-1)a_{n}}\cdot b^{a_{n}-a_{n+1}}= b^{ka_{n}-a_{n+1}}.
\end{equation}
By the assumption \eqref{eq:folge} on the sequence $(a_{n})_{n\geq 1}$, we have for any $\eta>0$
and sufficiently large $n\geq \widehat{n}(\eta)$ (or equivalently $x$ large enough)
\[
ka_{n}-a_{n+1}\geq (-k\rho-\eta)a_{n}.
\]
Together with \eqref{eq:huhn} we infer
\[
\max_{1\leq j\leq k}\Vert \zeta^{j} x\Vert\geq \Vert \zeta x\Vert \geq
b^{(-k\rho-\eta)a_{n}}\geq x^{-\rho-\frac{\eta}{k}}. 
\]
Thus, the approximation constant $\lambda_{k}(\zeta)$ restricted to 
pairs $(x,y)$ linearly dependent to some $\underline{x}_{n}$ is bounded above by $\rho+\eta/k$.
As we may choose $\eta$ arbitrarily small, this contradicts the assumption 
$\lambda_{k}(\zeta)>\max\{1,\rho\}$ as well.
\end{proof}

\begin{remark}
In fact we proved that for any $\epsilon>0$ and sufficiently large $n\geq \hat{n}(\epsilon)$,
any integer $x\in{[b^{a_{n}},b^{a_{n+1}})}$ not divisible by $b^{a_{n}}$ satisfies
$\max_{1\leq j\leq k} \Vert \zeta^{j}x\Vert \geq x^{-1-\epsilon}$.
Using the argument of case $T=1$ in Corollary~\ref{korollar} within the proof instead of the $T>1$ case,
this can be sharpened to $\max_{1\leq j\leq k} \Vert \zeta^{j}x\Vert\geq C_{0}x^{-1}$ 
with $C_{0}$ from Lemma~\ref{lemma2}.
\end{remark}

Eventually, we state two obvious conjectures concerning generalizations of Theorem~\ref{meinsatz}.
Both would imply a positive answer to Problem~\ref{problem1}.

\begin{conjecture}[Weak]
Let $k\geq 2,b\geq 2$ be integers and $\rho\geq 1/k$. 
Let $(a_{n})_{n\geq 1}$ be a strictly increasing sequence of positive 
integers with the property \eqref{eq:folge} and $\zeta$ as in \eqref{eq:zeter}.
Then $\lambda_{k}(\zeta)=\rho$.
\end{conjecture} 

\begin{conjecture}[Strong]
Let $k\geq 2, b\geq 2$ be integers and $\rho>0$. 
Let $(a_{n})_{n\geq 1}$ be a strictly increasing sequence of positive 
integers with the property \eqref{eq:folge} and $\zeta$ as in \eqref{eq:zeter}.
Then $\lambda_{k}(\zeta)=\max\{\rho,1/k\}$.
\end{conjecture}

The crucial point why the methods in the proof of Theorem~\ref{meinsatz} do not allow
for establishing better upper bounds for $\lambda_{k}(\zeta)$ in the context of 
Theorem~\ref{meinsatz} in case of $\rho< 1$, is that no 
extension of Lemma~\ref{lemma2} to this case seems available. 

\section{Consequences of Theorems~\ref{spektrum}, \ref{theo} for the dual constants}  \label{vier}
 
\subsection{Definition of the dual problem}  \label{dualproblem}
We conclude with applications to the dual constants $w_{k}(\zeta), \widehat{w}_{k}(\zeta)$.
For $\underline{\zeta}=(\zeta_{1},\ldots,\zeta_{k})\in{\mathbb{R}^{k}}$
the quantities $w_{k}(\underline{\zeta})$ and $\widehat{w}_{k}(\underline{\zeta})$
are respectively defined as the supremum of real numbers $\nu$ such that 
\[
\max_{0\leq j\leq k}\vert x_{j}\vert \leq X,  \qquad 
 0<\vert x_{0}+\zeta_{1} x_{1}+\cdots+\zeta_{k}x_{k}\vert \leq X^{-\nu},
\]
has a solution $(x_{0},x_{1},\ldots,x_{k})\in{\mathbb{Z}^{k+1}}$ for arbitrarily large $X$,
and for all $X\geq X_{0}$, respectively. Further let
\[
w_{k}(\zeta):=w_{k}(\zeta,\zeta^{2},\ldots,\zeta^{k}), \qquad
\widehat{w}_{k}(\zeta):=\widehat{w}_{k}(\zeta,\zeta^{2},\ldots,\zeta^{k}).
\]
Dirichlet's~Theorem yields for $\zeta$ not algebraic of degree $\leq k$ the estimates
\begin{equation} \label{eq:genauaso}
k\leq \widehat{w}_{k}(\zeta)\leq w_{k}(\zeta)\leq \infty.
\end{equation}
Moreover, it is not hard to check that \eqref{eq:monoton},~\eqref{eq:monodon} extend to
\begin{eqnarray} 
&\cdots&\leq \lambda_{3}(\zeta)\leq \lambda_{2}(\zeta)\leq \lambda_{1}(\zeta)=
w_{1}(\zeta)\leq w_{2}(\zeta)\leq w_{3}(\zeta)\leq\cdots,   \label{eq:dualugk}     \\
&\cdots&\leq \widehat{\lambda}_{3}(\zeta)\leq \widehat{\lambda}_{2}(\zeta)\leq \widehat{\lambda}_{1}(\zeta)=
\widehat{w}_{1}(\zeta)\leq \widehat{w}_{2}(\zeta)\leq \widehat{w}_{3}(\zeta)\leq\cdots. \nonumber
\end{eqnarray}
Khintchine's transference principle~\cite{khintchine} allows for a connection between the constants $\lambda_{k}$ 
and the constants $w_{k}$. Indeed, the original version is shown to be equivalent to
\begin{equation} \label{eq:kinschin}
\frac{w_{k}(\zeta)}{(k-1)w_{k}(\zeta)+k}\leq \lambda_{k}(\zeta)\leq \frac{w_{k}(\zeta)-k+1}{k}
\end{equation}
for all $\zeta\in{\mathbb{R}}$ in~\cite{bugg}. The analogue for the 
uniform constants $\widehat{w}_{k}, \widehat{\lambda}_{k}$ holds as well~\cite{buglau}, 
in fact a refined version due to German~\cite{german} can be written as
\begin{equation} \label{eq:dualdazu}
\frac{\widehat{w}_{k}(\zeta)-1}{(k-1)\widehat{w}_{k}(\zeta)}\leq 
\widehat{\lambda}_{k}(\zeta)\leq \frac{\widehat{w}_{k}(\zeta)-k+1}{\widehat{w}_{k}(\zeta)}.
\end{equation}
We should mention that \eqref{eq:kinschin} and \eqref{eq:dualdazu}
are valid in the much more general context of real vectors 
$\underline{\zeta}\in{\mathbb{R}^{k}}$ linearly independent over $\mathbb{Q}$ 
together with $\{1\}$. Consequently, the relations
\begin{equation} \label{eq:kin2}
\lambda_{k}(\zeta)=\frac{1}{k} \Longleftrightarrow w_{k}(\zeta)=k, \qquad 
\widehat{\lambda}_{k}(\zeta)=\frac{1}{k}\Longleftrightarrow \widehat{w}_{k}(\zeta)=k
\end{equation}
hold for any positive integer $k$ and any real number $\zeta$. Furthermore
\begin{equation} \label{eq:dualw}
\widehat{w}_{k}(\zeta)\leq 2k-1
\end{equation}
for $\zeta$ not algebraic of degree $\leq k$ is a consequence of~\cite[Theorem~2b]{schmidt1969}.

\subsection{Results for the dual problem} \label{rothw}

We say in advance that in the proofs in this section, we will assume $k\geq 2$,
since for $k=1$ the assertions follow from 
\eqref{eq:dirichlet},\eqref{eq:uniformity} and \eqref{eq:genauaso} if they are not trivial at all.
It will be convenient to use Roth's~Theorem~\cite{roth} at some places to exclude
the case that the numbers $\zeta$ involved are algebraic, in order to apply \eqref{eq:dualw}. 
It asserts that for irrational algebraic $\zeta$ we have $\lambda_{1}(\zeta)=1$.
Note also that Corollary~\ref{korollar} yields the equivalence of 
the inequalities $\lambda_{k}(\zeta)>1$ and $w_{1}(\zeta)=\lambda_{1}(\zeta)>2k-1$.

First we want to point out a
consequence of Theorem~\ref{theo}. It allows for improving the bound in \eqref{eq:dualw}
provided that $\lambda_{1}(\zeta)$ is sufficiently large.

\begin{theorem}  \label{neuko}
Let $k$ be a positive integer and $\zeta$ an irrational real number. In case of $w_{1}(\zeta)\geq k$, 
which is in particular true if $\lambda_{k}(\zeta)>1$, the equalities
\begin{equation}  \label{eq:zwischen}
\widehat{w}_{1}(\zeta)=1, \quad \widehat{w}_{2}(\zeta)=2,\quad \ldots,\quad \widehat{w}_{k}(\zeta)=k
\end{equation}
hold. In case of $k-1<w_{1}(\zeta)<k$, we have the inequalities
\begin{equation}  \label{eq:schwamm} 
k\leq \widehat{w}_{k}(\zeta)\leq \min\left\{\frac{w_{1}(\zeta)}{w_{1}(\zeta)-k+1},2k-1\right\}.
\end{equation}
\end{theorem}

\begin{proof}
Note that $w_{1}(\zeta)=\lambda_{1}(\zeta)$ by \eqref{eq:dualugk}. Thus Theorem~\ref{theo} implies
$\widehat{\lambda}_{j}(\zeta)=1/j$ for $1\leq j\leq k$, and \eqref{eq:zwischen} follows from \eqref{eq:kin2}.
Concerning the non-trivial right hand side inequality in
\eqref{eq:schwamm}, note that by assumption and Roth's~Theorem $\zeta$ is transcendental, 
such that the upper bound $2k-1$ is obtainded from \eqref{eq:dualw}. The remaining upper bound in \eqref{eq:schwamm}
follows from \eqref{eq:dualdazu} and Theorem~\ref{theo} via
\[
\frac{\widehat{w}_{k}(\zeta)-1}{(k-1)\widehat{w}_{k}(\zeta)}\leq \widehat{\lambda}_{k}(\zeta)\leq 
\max\left\{\frac{1}{k},\frac{1}{\lambda_{1}(\zeta)}\right\}=
\max\left\{\frac{1}{k},\frac{1}{w_{1}(\zeta)}\right\}=\frac{1}{w_{1}(\zeta)}
\]
by elementary rearrangements.
\end{proof}

One readily sees \eqref{eq:schwamm} is an improvement to \eqref{eq:dualw} 
in case of $w_{1}(\zeta)>k-1/2$.

In combination with certain results established before, Theorem~\ref{neuko} allows for finally determining
{\em all} the classical approximation constants introduced in the Sections~\ref{intro} and~\ref{dualproblem} 
for Liouville numbers.

\begin{corollary}
Let $\zeta$ be a Liouville number, i.e., a real number which satisfies $\lambda_{1}(\zeta)=\infty$. 
Then for any $k\geq 1$ we have
\[
\lambda_{k}(\zeta)=\infty, \quad \widehat{\lambda}_{k}(\zeta)=\frac{1}{k},
\quad w_{k}(\zeta)=\infty, \quad \widehat{w}_{k}(\zeta)=k.
\]
\end{corollary}

\begin{proof}
The assertion on $\widehat{w}_{k}(\zeta)$ follows from \eqref{eq:zwischen}.
The claim on $\lambda_{k}(\zeta)$ was established in Corollary~2 in~\cite{bug},
as carried out preceding Corollary~\ref{bestens}. The assertion
on $\widehat{\lambda}_{k}(\zeta)$ follows either from Theorem~\ref{theo}
or \eqref{eq:kin2}, and the assertion on $w_{k}(\zeta)$ is due to \eqref{eq:dualugk}.
\end{proof} 

Recall $\lambda_{k}(\zeta)=\infty$ for any Liouville number $\zeta$ and all $k\geq 1$, 
as already mentioned in Section~\ref{intro}.
In particular, for any $k\geq 1$ and any parameter $\theta>0$ the estimate
$\max_{1\leq j\leq k}\Vert \zeta^{j} x\Vert< \theta x^{-1}$
has arbitrarily large integer solutions $x$.
Together with the case $T=1$ in Corollary~\ref{korollar},
we stem a new criterion for a number to be a Liouville number. 

\begin{theorem}  \label{kommtschonoch}
An irrational real number $\zeta$ is a Liouville number
if and only if for any positive integer $k$, the estimate
\[
\max_{1\leq j\leq k} \Vert \zeta^{j} x\Vert < C_{0}(k,\zeta)\cdot x^{-1}
\]
with $C_{0}$ defined in Lemma~{\upshape\ref{lemma2}}, has an integer solution $x=x(k,\zeta)>0$.
\end{theorem}

Finally, in the case $\lambda_{k}(\zeta)>1$, Theorem~\ref{spektrum} implies the right hand side inequality 
in Khintchine's transference principle \eqref{eq:kinschin} and a criterion for equality. 

\begin{theorem}
Let $k$ be a positive integer and $\zeta$ be a real number with $\lambda_{k}(\zeta)>1$, 
or equivalently $w_{1}(\zeta)=\lambda_{1}(\zeta)>2k-1$.
We have equality in the right hand side inequality of {\upshape(\ref{eq:kinschin})} if and only if
\[
w_{1}(\zeta)=w_{2}(\zeta)=\cdots=w_{k}(\zeta).
\]
\end{theorem}

\begin{proof}
In view of \eqref{eq:dualugk} and Theorem~\ref{spektrum}, for $\lambda_{k}(\zeta)>1$ indeed
\[
\lambda_{k}(\zeta)= \frac{\lambda_{1}(\zeta)-k+1}{k}=\frac{w_{1}(\zeta)-k+1}{k}\leq \frac{w_{k}(\zeta)-k+1}{k}.
\]
Clearly the given equivalence is valid.
\end{proof}

The ''only if'' statement is the contribution of Theorem~\ref{spektrum}, the ''if'' part 
can be inferred from \eqref{eq:holds} and \eqref{eq:kinschin} without any restriction on $\lambda_{k}(\zeta)$ 
as already implicitly carried out in the proof of~\cite[Theorem~2]{bug}.

\vspace{1cm}
The author warmly thanks Yann Bugeaud and the anonymous referee for remarks that 
helped me to improve the original version, in particular the presentation.

\end{document}